\documentclass{amsart}
\usepackage{graphicx}
\usepackage{enumerate}
\usepackage{fullpage}
\usepackage{subfigure}
\newtheorem{theorem}{Theorem}[section]

\newtheorem{proposition}[theorem]{Proposition}
\theoremstyle{definition}

\theoremstyle{remark}

\numberwithin{equation}{section}

\newcommand{\set}[1]{\left\{#1\right\}}
\newcommand{\comb}[2]{\left( \begin{array}{c} #1 \\ #2 \end{array} \right)}

\newcommand{\C}{\mathbb C}

\newcommand{\Z}{\mathbb Z}

\newcommand{\BB}{{\mathcal B}}
\newcommand{\PP}{{\mathcal P}}
\newcommand{\HH}{{\mathcal H}}
\newcommand{\EE}{{\mathcal E}}

\begin{document}

\title[]
{On the leading eigenvalue of transfer operators of the Farey map with real temperature}%
\author{S. Ben Ammou} 
\address{Faculty of Science, Computational Mathematics Laboratory, University of Monastir, Monastir 5000, Tunisia}
\author{C. Bonanno} 
\address{Dipartimento di Matematica, Universit\`a di Pisa, Pisa, Italy}
\email{bonanno@dm.unipi.it}
\author{I. Chouari}
\address{Faculty of Science, Computational Mathematics Laboratory, University of Monastir, Monastir 5000, Tunisia}
\author{S. Isola}
\address{Dipartimento di Matematica e Informatica, Universit\`a di Camerino, Camerino (MC), Italy}
\thanks{The second author is partially supported by ``Gruppo Nazionale per l'Analisi Matematica, la Probabilit\`a e le loro Applicazioni (GNAMPA)'' of Istituto Nazionale di Alta Matematica (INdAM), Italy.}
\maketitle

\begin{abstract}
We study the spectral properties of a family of generalized transfer operators associated to the Farey map. We show that when acting on a suitable space of holomorphic functions, the operators are self-adjoint and the positive dominant eigenvalue can be approximated by means of the matrix expression of the operators.
\end{abstract}
\maketitle

\section{Introduction} \label{inizio}
Let $F:[0,1]\to [0,1]$ be the \emph{Farey map} defined by
\begin{equation} \label{farey}
F(x)=\left\{
\begin{array}{ll}
\frac{x}{1-x} & \mbox{if }\ 0\le x\le \frac{1}{2}\\[0.3cm]
\frac{1-x}{x} & \mbox{if }\ \frac{1}{2} \le x \le 1
\end{array} \right.
\end{equation}
Its name comes from the relations with the \emph{Farey fractions} as studied for example in \cite{CM}, and it is a particularly interesting map to study from the point of view of Ergodic Theory, also for applications to number theory (see e.g. \cite{Is2}).

The Farey map is an example of so-called \emph{intermittent maps} with an infinite absolutely continuous invariant measure (see e.g. \cite{GW,LSV}). A map is \emph{intermittent} if roughly speaking, its orbits alternate regular and chaotic behavior. In the maps $T:[0,1]\to[0,1]$ studied in \cite{GW,LSV}, this is caused by the existence of a non-hyperbolic fixed point of the map $T$ at the origin, that is $T(0)=0$ and $T'(0)=1$, with $T'(x)>1$ for all $x\not= 0$. Many ergodic properties fail for intermittent maps, for example the decay of correlations occur at a polynomial rate instead of the exponential rate typical of expanding maps, for which $T'(x)\ge c >1$ for all $x$. Moreover if $T'(x)= 1+\alpha x^{\alpha-1}+o(x^{\alpha-1})$ as $x\to 0$ for $\alpha>0$, then the unique invariant measure $\mu$, which is absolutely continuous with respect to the Lebesgue measure, is finite if and only if $\alpha<2$. Hence for $\alpha\ge 2$ we have to apply Infinite Ergodic Theory (\cite{AA}), and the analysis is more delicate. The Farey map corresponds to the case $\alpha=2$ with invariant measure $\mu(dx) = \frac 1x dx$.

For intermittent maps with infinite invariant measures, one of the problem is the study of the transfer operators associated to the map. This is a family of linear operators depending on a parameter $q\in \C$, that when defined on a suitable space of functions are useful in the study of ergodic properties. For example for $q=1$ one obtains information about the existence of invariant measures, ergodicity and mixing (\cite{Bal}). However the situation is much more delicate for systems with an infinite invariant measure.

In this paper we continue the study initiated in \cite{BGI,BI} of the transfer operators associated to the Farey map \eqref{farey}. In particular we consider the family of \emph{signed
generalized transfer operators} $\PP_q^{\pm}$ whose action on a function $f(x) :
[0,1]\to \C$ is given by a weighted sum over the values of $f$ on
the set $F^{-1}(x)$, namely letting
\[
(\PP_{0,q} f)(x) := \left(\frac{1}{x+1}\right)^{2q} \ f\left(
\frac{x}{x+1}\right)
\]
\[
(\PP_{1,q} f)(x) := \left(\frac{1}{x+1}\right)^{2q} \ f\left(
\frac{1}{x+1}\right)
\]
we have
\begin{equation} \label{transfer-farey}
(\PP_q^{\pm} f)(x) := (\PP_{0,q} f)(x) \pm (\PP_{1,q} f)(x)
\end{equation}
Since the operators $\PP_{q}^{\pm}$ are defined by multiplication and composition by real analytic maps which extend to holomorphic maps on a neighbourhood of the interval $[0,1]$, it is natural to consider the action of $\PP_{q}^{\pm}$ on the space of holomorphic functions on the open domain 
\begin{equation} \label{ball}
B:= \set{x\in \C\ :\ \left| x-\frac 1 2 \right| < \frac 1 2}.
\end{equation}
The study of the spectral properties of the operators $\PP_{q}^{\pm}$ has been considered in \cite{BGI,Is,Pr} for real values of $q$ and in \cite{BI} for complex $q$ with $\Re(q)>0$. In particular in \cite{BI}, the operators $\PP_q^{\pm}$ have been used for the extension of the results in \cite{LZ,Ma} to a two-variable Selberg zeta function. 

Here we consider the case $q\in (0,\infty)$ and pursue the analysis of the spectral properties of $\PP_q^{\pm}$ when acting on a suitable Hilbert space. In Section \ref{sec:spectrum} we first revisit results from \cite{BGI,BI}, amending in Proposition \ref{spazi} and Corollary \ref{new-correct} some results from \cite{BGI}. Then in Section \ref{sec:matrix} we introduce an infinite matrix associated to $\PP_q^{\pm}$, and use a result from \cite{K1} to obtain an algorithm for the determination of the leading eigenvalue of $\PP_{q}^{\pm}$. Compared to similar results in \cite{Pr}, our approach is direct, in the sense that we don't need to use an induced map, and the eigenvalue can be found with arbitrary accuracy.

\section{The spectrum of $\PP_{q}^{\pm}$ for real positive $q$} \label{sec:spectrum}

In \cite{BGI} the spectral properties of $\PP_{q}^{\pm}$ are studied on the Hilbert space $\HH_{q}$ of complex functions which can be represented in terms of the integral transform
\[
L^{2}(m_{q}) \ni \phi(t) \mapsto {\BB}_q [\phi](x):= \frac{1}{x^{2q}}\ \int_0^\infty
e^{-\frac{t}{x}}\, e^t\, \phi (t)\, m_q(dt)
\]
where $m_{q}$ is the absolutely continuous measure on $(0,\infty)$
defined as 
\[ 
m_q(dt)= t^{2q-1}\, e^{-t}\, dt.
\]
Hence the operators $\PP_{q}^{\pm}$ are studied on the space
\[
\HH_{q} := \set{ f(x) = {\BB}_q [\phi](x)\ :\ \phi(t) \in L^{2}(m_{q})}
\]
endowed with the inner product
\[
(f_1,f_2) := \int_0^\infty \varphi_1(t)\, \overline{\varphi_2
(t)}\, m_q(dt) \quad \mbox{if }\ f_i=\BB_q [\varphi_i]
\]
which makes $\HH_{q}$ a Hilbert space. In \cite{BI} it is shown that $\HH_q$ is naturally continuously embedded in $\HH(B)$, where $B$ is as in \eqref{ball} and $\HH(B)$ is the space of holomorphic functions endowed with the standard topology induced by the family of supremum norms on compact subsets of $B$. A first result from \cite{BGI} states

\begin{theorem}[\cite{BGI}] \label{lilla}
For $q\in (0,\infty)$ the space $\HH_q$ is invariant for
$\PP_q^{\pm}$ and $\PP_q^\pm:\HH_q\to \HH_q$ are isomorphic to self-adjoint compact perturbations of the multiplication operator $M:{\rm L^2}(m_q)\to {\rm L^2}(m_q)$
given by
$$(M\varphi) (t)= e^{-t} \, \varphi(t)$$ More specifically
$$\PP_q^{\pm}\, \BB_q\, [\varphi] = \BB_q\, [P_{q}^\pm  \varphi]$$
where $P_{q}^\pm=M\pm N_{q}$ and $N_{q}: {\rm L^2}(m_q)\to {\rm L^2}(m_q)$ is
the symmetric integral operator given by
$$(N_{q} \varphi) (t) = \int_0^\infty \frac{J_{2q-1} \left( 2 \sqrt{st}
\right)}{(st)^{q- 1/2}}\ \varphi(s) \, m_q(ds)$$ where $J_p$ denotes
the Bessel function of order $p$.
\end{theorem}

Moreover we remark that in \cite{BI} it is shown that the restriction of the operators $\PP_{q}^{\pm}$ to the space $\HH_{q}$ is natural for all complex $q$ with $\Re(q)>0$, since 

\begin{proposition}[\cite{BI}] \label{autofunz}
The eigenfunctions $f$ of $\PP_{q}^{\pm}$ with eigenvalue $\lambda \in \C \setminus
[0,1)$ have the form
\[
f(x) = \frac{c\, \lambda^{\frac 1 x}}{x^{2q}} +
\frac{\Gamma(2q-1)}{\Gamma(2q)}\, \frac b x + \BB_q[\phi]
\]
with $c,b \in \C$ and $\phi \in L^2(m_{_{\Re(q)}})$ with $\phi(0)$ finite and $\phi(t) = \phi(0) + O(t)$ as $t\to 0^{+}$. Moreover if $\lambda
\not= 1$ then $b=0$. Instead if $\lambda =1$ then $b = f(1)$, and $b=f(1)=0$ if $f$ is an eigenfunction of $\PP_{q}^{-}$. Finally the function $\phi \in L^2(m_{_{\Re(q)}})$ is such that $\BB_q[\phi]$ is bounded as $x\to 0$.
\end{proposition}
 
The study of the spectral properties of $\PP_{q}^{\pm}$ on the space $\HH_{q}$, is then reduced to study of the spectral properties of the operators $P_{q}^{\pm} = M \pm N_{q}$ defined in Theorem \ref{lilla} on the space $L^{2}(m_{q})$. To this aim we recall some properties of the Hilbert space $L^{2}(m_{q})$. First of all the measure $m_q(dt)$ is finite, indeed
\[
\int_0^\infty\ m_q(dt)=\Gamma (2q)
\]
Second, the (generalised) Laguerre polynomials
$L_n^{2q-1}(t)$, $n\ge 0$, defined by
\begin{equation}\label{laguerres}
e_n(t):=L_n^{2q-1}(t) = \sum_{m=0}^n\ \comb{n+2q-1}{n-m} \
\frac{(-t)^m}{m!}
\end{equation}
form a complete orthogonal system in $L^2 (m_q)$, with
\begin{equation}\label{normaliz}
(e_n, e_m) = \frac{\Gamma(n+2q)}{n!}\ \delta_{n,m}
\end{equation}
The following families of functions are introduced in \cite{BGI}. Consider the independent
family of functions $f_n(t):= \frac{t^n}{n!}$ which satisfy
\begin{equation} \label{emme-enne}
N_{q} f_n  = M e_n \qquad N_{q} e_n= M f_n,
\end{equation} 
and let
\[
\ell_n^{\pm}(t):= e_n(t)\pm f_n(t), \ \qquad \ \zeta_n^{\pm}(t) :=
e^{-t}( e_n(t) \pm f_n(t))
\]
We have
\begin{proposition} \label{spazi}
Let $H^{\pm} := Span \set{\ell_{n}^{\pm}}_{n\ge 0}$ and $\EE^{\pm} := Span \set{\zeta_{n}^{\pm}}_{n\ge 0}$. Then
\begin{enumerate}[(i)]
\item $H^+ \cap H^- =\set{0}$ and $\EE^+ \cap \EE^- =\set{0}$; 

\item ${\rm Ker}\, P_{q}^{\pm} = H^{\mp}$;

\item $H^{\pm} = (\EE^{\mp})^{\bot}$, hence $L^{2}(m_{q}) = H^{+} \oplus \EE^{-} = H^{-} \oplus \EE^{+}$;

\item $P_q^\pm$ are positive operators;

\item if $P_{q}^{\pm} \phi = \lambda \phi$ for some $\lambda\not= 0$ then $\phi \in \EE^{\pm}$;

\item the spectrum of $P_{q}^{\pm}$ in $L^{2}(m_{q})$ is real and consists of the absolutely continuous spectrum $\sigma_{ac}(P_{q}^{\pm}) = [0,1]$ and of the point spectrum $\sigma_{p}(P_{q}^{\pm})$.
\end{enumerate}
\end{proposition}

\begin{proof}
Points (i)-(ii)-(vi) are proved in \cite{BGI}.

(iii) We first show that $H^{+} \subseteq (\EE^{-})^{\bot}$. It follows from
$$
(\ell_{k}^{+}, \zeta_{n}^{-}) = (\ell_{k}^{+}, Me_{n} - Mf_{n}) = (\ell_{k}^{+}, Me_{n} - N_{q}e_{n}) = (\ell_{k}^{+}, P_{q}^{-}e_{n}) = (P_{q}^{-} \ell_{k}^{+}, e_{n}) =0
$$
for any $n,k \ge 0$, where we have used (\ref{emme-enne}), the self-adjointness of $P_{q}^{-}$ and (ii). Finally we show $ (\EE^{-})^{\bot} \subseteq H^{+}$. Let indeed $\phi \in (\EE^{-})^{\bot}$ and write its expression in the basis $\set{e_{k}}$. Then for any $n\ge 0$
$$
0 = (\phi, \zeta_{n}^{-}) = \sum_{k}\, \phi_{k}\, (e_{k}, \zeta_{n}^{-}) = \sum_{k}\, \phi_{k}\, (P_{q}^{-} e_{k}, e_{n}) = (P_{q}^{-} \phi, e_{n}) 
$$
where we have argued as above. Then $\phi \in {\rm Ker}\, P_{q}^{-} = H^{+}$ by (ii). The same proof works to show that $H^{-} = (\EE^{+})^{\bot}$. Now (iii) follows by the standard theory of Hilbert spaces.

(iv) From $e_k=\frac 12 (\ell_k^+ + \ell_k^-)$ one deduces that $L^2(m_q)=H^+ + H^-$, even if the spaces are not mutually orthogonal. From this, writing $\phi = \phi_+ +\phi_-$, with $\phi_\pm \in H^\pm$, it follows that
$$(\phi, P_q^\pm \phi) = (\phi_+ + \phi_-, P_q^\pm \phi_\pm) = (\phi_\pm, P_q^\pm \phi_\pm)\, ,$$
where we have used (ii) and the self-adjointness of $P_q^\pm$. Moreover, from \eqref{emme-enne} it follows that $P_q^\pm \phi_\pm = 2M\phi_\pm$, hence
$$(P_q^\pm\phi_\pm,  \phi_\pm)=(\phi_\pm, P_q^\pm \phi_\pm) = \int_0^\infty\, 2\, e^{-t}\, |\phi_\pm(t)|^2 \, m_q(dt) \ge 0$$

(v) This follows from the self-adointness of $P_{q}^{\pm}$ together with the orthogonal decompositions of $L^{2}(m_{q})$ given in (iii). Let $\phi$ satisfy $P_{q}^{+} \phi = \lambda \phi$ for some $\lambda\not= 0$, then it is enough to show that for each $\psi \in H^-$ we have $(\phi, \psi)=0$. This follows from
\[
\lambda (\phi, \psi) = (\lambda \phi, \psi) = (P_{q}^{+} \phi, \psi) = (\phi, P_{q}^{+} \psi) = 0
\]
where in the last equality we have used (ii).
\end{proof}

Corollary 2.14 in \cite{BGI} has then to be changed into

\begin{proposition} \label{new-correct}
In $L^{2}(m_{q})$ the point spectrum of $P_{q}^{\pm}$ contains $\lambda=0$ with infinite multiplicity and real eigenvalues contained in the set $[0, 1+\gamma^{2q}]$, where $\gamma = \frac{\sqrt{5}-1}{2}$. The eigenvalues not embedded in $[0,1]$ are isolated and have finite multiplicity. 
\end{proposition}

\begin{proof}
The statements follow from the standard perturbation theory of operators (see e.g. \cite{K2}). In particular if we think of $P_{q}^{\pm} = M\pm N_{q}$ as a perturbation of the compact operator $N_{q}$, then it follows that the resolvent of $P_{q}^{\pm}$ contains the set $\set{z\in \C\ :\ dist(z,\sigma(N_{q}))>1}$. In \cite[Proposition 2.15]{BGI} it is shown that the spectrum of $N_{q}$, which consists only of eigenvalues, is
$$
\sigma(N_{q}) = \set{0} \cup \set{(-1)^{k}\, \gamma^{2q+2k}}_{k\ge 0}
$$
hence, since $P_q^\pm$ are positive operators, it immediately follows that the spectrum of $P_{q}^{\pm}$ is contained in the set $[0, 1+\gamma^{2q}]$. 

Moreover, if we instead think of $P_{q}^{\pm}$ as a compact perturbation of the operator $M$, then it follows that $P_{q}^{\pm}$ and $M$ have the same essential spectrum, the interval $[0,1]$.
\end{proof}

\section{The matrix approach} \label{sec:matrix}

We now give further results on the point spectrum of $P_{q}^{\pm}$. The results from \cite{Pr}, proved using theoretical and numerical techniques, imply that $\sigma_{p}(P_{q}^{+})$ consists of a single eigenvalue $\lambda_{q}\in (1,2]$ for $q\in [0,1)$, and is empty for $q\ge 1$. Moreover general results from \cite{PS} imply that $\lambda_{q}$ approaches 1 like $(1-q) \log \frac{1}{1-q}$ as $q$ approaches 1 from below. 

To study the behaviour of the eigenvalues, we consider the matrix formulation of $P_{q}^{\pm}$ in terms of the orthogonal basis $\set{e_{k}}$ of $L^{2}(m_{q})$ introduced in (\ref{laguerres}). Let $\phi \in L^{2}(m_{q})$ be written as
$$
\phi(t) = \sum_{n=0}^{\infty}\, \phi_{n}\, e_{n}(t) \qquad \text{with} \quad (\phi, e_{n}) = \phi_{n} \, (e_{n}, e_{n}) = \phi_{n}\, \frac{\Gamma(n+2q)}{n!}\, ,
$$
by \eqref{normaliz}, then $\phi$ is an eigenfunction of $P_{q}^{\pm}$ with eigenvalue $\lambda$ if and only if
$$
\left( P_{q}^{\pm} \phi, e_{k} \right) = \lambda\, (\phi, e_{k}) = \lambda\, \phi_{k}\, \frac{\Gamma(k+2q)}{k!} \qquad \forall\, k\ge 0\, .
$$
Moreover it has been shown in \cite{BGI} that
\[
c^{\pm}_{nk} := \left( P_{q}^{\pm} e_{n}, e_{k} \right) = \frac{\Gamma(n+2q)\, \Gamma(k+2q)}{2^{n+k+2q}}\, \sum_{m=0}^{\min\set{n,k}}\, \frac{1\pm(-1)^{m}}{\Gamma(m+2q)\, m!\, (n-m)!\, (k-m)!} > 0
\]
for all $n,k \ge 0$. Hence we obtain that
\begin{equation} \label{matrix-form}
P_{q}^{\pm} \phi = \lambda \phi \quad \Leftrightarrow \quad C^{\pm}\, \Phi = \lambda D\, \Phi \quad \Leftrightarrow \quad A^{\pm}\, \Phi = \lambda\, \Phi\, ,
\end{equation}
where $C^{\pm}$ and $D$ are symmetric infinite matrices given by
$$
C^{\pm} = (c^{\pm}_{kn})_{n,k\ge 0} \qquad D = \text{diag} \left( \frac{\Gamma(n+2q)}{n!} \right)\, ,
$$
$A^{\pm}$ is the infinite non-symmetric matrix
\begin{equation} \label{matrice-2}
A^{\pm} = (\alpha^{\pm}_{kn})_{n,k\ge 0}, \qquad \alpha^{\pm}_{kn} = \frac{k!\, \Gamma(n+2q)}{2^{n+k+2q}}\, \sum_{m=0}^{\min\set{n,k}}\, \frac{1\pm (-1)^{m}}{\Gamma(m+2q)\, m!\, (n-m)!\, (k-m)!} > 0\, ,
\end{equation}
and $\Phi$ is the infinite vector
\begin{equation} \label{cond-vett}
\Phi = \left( \phi_{n} \right)_{n\ge 0} \qquad \text{such that} \quad \sum_{n=0}^{\infty}\, |\phi_{n}|^{2}\, \frac{\Gamma(n+2q)}{n!} < \infty\, .
\end{equation}
By (\ref{matrix-form}) we can study the eigenvalue problem for the matrix $A^{\pm}$.

\begin{proposition} \label{p-matrix}
All minor determinants of second-order of the matrix $A^{\pm}$ are non-negative.
\end{proposition}

\begin{proof}
See the Appendix.
\end{proof}

\noindent We can now apply the main result of \cite{K1} to the bounded operators $A^{\pm}$ to obtain
\begin{proposition}[\cite{K1}] \label{autoval-gen}
Let $A^{\pm}_{N}$ be the $N\times N$ north-west corner approximation of the matrix $A^{\pm}$ and let $\lambda_{N}$ and $\phi^{N}$ be the dominant eigenvalue and eigenvector with $\phi^{N}_{0}=1$. Then the sequence $\set{\lambda_{N}}_{N}$ is increasing and the sequences $\set{\phi^{N}_{k}}_{N}$ with fixed $k$ are non-decreasing. Then $\lambda := \lim_{_{N\to \infty}} \lambda_{N}$ is an eigenvalue of $A^{\pm}$ and $\Phi = (\phi_{k})_{k\ge 0}$, defined as $\phi_{k} := \lim_{_{N\to \infty}} \phi^{N}_{k}$, is an eigenvector such that $A^{\pm} \Phi = \lambda \Phi$.
\end{proposition}

From the theory of finite positive matrices it also follows that the eigenvalues $\lambda_{N}$ are positive and that for each $k$ it holds $\phi^{N}_{k} >0$. Hence the eigenvalue $\lambda$ and the eigenvector $\Phi$ are positive. We remark that condition (\ref{cond-vett}) for $\Phi$ is guaranteed by results in \cite{Pr}. However our method doesn't imply the non-existence of other eigenvectors, which is known again by \cite{Pr}.

A first consequence is the following proposition.

\begin{proposition} \label{primi-ris}
Let $\lambda^{\pm}_{q}$ be the eigenvalues of $A^{\pm}$ as obtained in Proposition \ref{autoval-gen} as functions of the real parameter $q$, and $\Phi^{\pm}$ the correspondent positive eigenvectors. Then
\begin{enumerate}[(i)]
\item $\phi^{+}_{0} = 1$, $\phi^{+}_{1} = \frac 12$;

\item $\phi^{-}_{0} = 0$, $\phi^{-}_{1} = 1$;

\item if $\phi^{+}_{k} \le \frac 12$ for all $k\ge 1$ then $\lambda^{+}_{q} \le 1+ 2^{-2q}$;

\item if $\phi^{-}_{k} \le 1$ for all $k\ge 1$ then $\lambda^{-}_{q} \le 1$.
\end{enumerate}
\end{proposition}

\begin{proof}
(i) follows from $\alpha^{+}_{0n} = 2 \alpha^{+}_{1n}$. (ii) follows from the construction of $\Phi^{-}$. (iii) follows from
$$
\lambda^{+}_{q} = \lambda^{+}_{q} \phi^{+}_{0} = \alpha^{+}_{00} \phi^{+}_{0} + \sum_{n=1}^{\infty}\, \alpha^{+}_{0n}\, \phi^{+}_{n} \le 2^{1-2q} + \frac 12\ \sum_{n=1}^{\infty}\, 2^{-n-2q+1}\, \frac{\Gamma(n+2q)}{\Gamma(2q)\, n!} = 1+2^{-2q}
$$
where in the last equality we have used the identity (see e.g. \cite[vol.I, pag.101]{Er})
$$
(1+z)^{a} = F(-a, b;b; -z) = \sum_{n=0}^{\infty}\, \frac{\Gamma(n-a)}{\Gamma(-a)\, n!}\, (-z)^{n}
$$
for the hypergeometric functions. In the same way (iv) follows from
$$
\lambda^{-}_{q} = \lambda^{-}_{q} \phi^{-}_{1} = \sum_{n=1}^{\infty}\, \alpha^{-}_{1n}\, \phi^{-}_{n} \le \sum_{n=1}^{\infty}\, 2^{-n-2q}\, \frac{\Gamma(n+2q)}{\Gamma(2q+1)\, (n-1)!} = 1
$$
\end{proof}

\vskip 0.2cm

Proposition \ref{autoval-gen} implies that we can numerically find curves which are a lower bound for the eigenvalues $\lambda^{\pm}_{q}$. If we fix $N=50$ in Proposition \ref{autoval-gen} we obtain curves $\lambda^{\pm}_{50}(q)$ which are a lower bound for $\lambda^{\pm}_{q}$. This is shown in Figure \ref{confronto}. The convergence of $\lambda^{\pm}_{N}(q)$ to $\lambda^{\pm}_{q}$ is quite fast, as shown for $\lambda^{+}_{N}$ in Figure \ref{convergence} for the cases $q=\frac 13$ and $q=0.95$ with $N=1,\dots,50$. Moreover the sequences $\{\lambda^{\pm}_{N}(q)\}$ are increasing, hence we can determine $\lambda^{\pm}_{q}$ with arbitrary accuracy by increasing $N$.

Results from Proposition \ref{primi-ris} and Figure \ref{confronto} suggest that there is an eigenvalue bigger than 1 only for the operators $P^+_q$ with $q\in (0,1)$, whereas the entire spectrum of $P^-_q$ is contained in $[0,1]$. Hence we now restrict to the operators $P^+_q$ and numerically analyze the convergence of the sum in \eqref{cond-vett} for the eigenvector $\Phi$ corresponding to the dominant eigenvalue. 

By Proposition \ref{autoval-gen}, for any fixed value $N$ we find $\phi^N$, which gives an approximation of the first $N$ components of $\Phi$. Then we consider the sums
\[
S_N(k):= \sum_{n=0}^{k}\, |\phi_{n}^N|^{2}\, \frac{\Gamma(n+2q)}{n!}\qquad k=0,\dots,N-1\, .
\]
and plot the functions $k\mapsto S_N(k)$  for different values of $N$, in particular $N=50, 100, 130$ in Figure \ref{fig-sums}. Using different $N$ we want to measure the changes in the components of the eigenvector. What we find is that for $q$ small, the changes are small (for $q=0.3$ the curves $S_N(k)$ are close to each other), and increasing $q$ the changes are bigger. This is reasonable, since we know that for $q=1$ there is no convergence of the sum $S_N(N-1)$ as $N\to \infty$.

Therefore, we interpret the curves in Figure \ref{fig-sums} as an indication that for $q= 0.3$, $q=0.5$ and $q=0.95$, i.e for $q<1$, there is a convergence to a real eigenvector of the infinite matrix, but this does not happen for $q=1$.

\appendix

\section*{Appendix. Proof of Proposition \ref{p-matrix}}

Let's first consider $A^{+}$. Using \eqref{matrice-2}, we have to show that
\[
\alpha^{+}_{k,n}\, \alpha^{+}_{k+1, n+1}\, - \alpha^{+}_{k+1,n}\, \alpha^{+}_{k,n+1}\, \ge 0
\]
for all $k,n\ge 0$. We show the proof in the case $n$ even and $k\ge n+1$. The other cases and the computations for $A^-$ are similar.

\underline{Case $n$ even and $k\ge n+1$.} We have $\min\set{k,n} = n$, $\min\set{k,n+1}= n+1$ and $1+(-1)^{n+1}=0$. Hence we find
\[
\alpha^{+}_{k, n} = \frac{k!\, \Gamma(n+2q)}{2^{n+k+2q}}\, \sum_{m=0}^{n}\, \frac{1+ (-1)^{m}}{\Gamma(m+2q)\, m!\, (n-m)!\, (k-m)!} 
\]
\[
\alpha^{+}_{k+1, n+1} = \frac{(k+1)!\, \Gamma(n+1+2q)}{2^{n+k+2+2q}}\, \sum_{m=0}^{n}\, \frac{1+ (-1)^{m}}{\Gamma(m+2q)\, m!\, (n+1-m)!\, (k+1-m)!} 
\]
\[
\alpha^{+}_{k, n+1} = \frac{k!\, \Gamma(n+1+2q)}{2^{n+k+1+2q}}\, \sum_{m=0}^{n}\, \frac{1+ (-1)^{m}}{\Gamma(m+2q)\, m!\, (n+1-m)!\, (k-m)!} 
\]
\[
\alpha^{+}_{k+1, n} = \frac{(k+1)!\, \Gamma(n+2q)}{2^{n+k+1+2q}}\, \sum_{m=0}^{n}\, \frac{1+ (-1)^{m}}{\Gamma(m+2q)\, m!\, (n-m)!\, (k+1-m)!} 
\]
Then we can define the factors
\[
C_{k,n}:= \frac{k!\ (k+1)! \Gamma(n+2q)\, \Gamma(n+1+2q)}{2^{2n+2k+2+4q}} \ge 0
\]
\[
C_{m,\tilde m}:= \frac{(1+ (-1)^{m})\, (1+ (-1)^{\tilde m})}{\Gamma(m+2q)\Gamma(\tilde m+2q) m!\, \tilde m!} \ge 0
\]
and write
\[
\begin{aligned}
& \alpha^{+}_{k,n}\, \alpha^{+}_{k+1, n+1}\, - \alpha^{+}_{k+1,n}\, \alpha^{+}_{k,n+1} = \\[0.2cm]
& =C_{k,n} \sum_{m=0}^{n}\, \sum_{\tilde m=0}^{n}\, \frac{(1+ (-1)^{m})\, (1+ (-1)^{\tilde m})}{\Gamma(m+2q)\Gamma(\tilde m+2q) m!\, \tilde m!}\ \frac{1}{(n-m)!(k-m)! (n+1-\tilde m)! (k+1-\tilde m)!} - \\[0.2cm]
& - C_{k,n} \sum_{m=0}^{n}\, \sum_{\tilde m=0}^{n}\, \frac{(1+ (-1)^{m})\, (1+ (-1)^{\tilde m})}{\Gamma(m+2q)\Gamma(\tilde m+2q) m!\, \tilde m!}\ \frac{1}{(n-m)!(k+1-m)! (n+1-\tilde m)! (k-\tilde m)!} =\\[0.2cm]
& = C_{k,n} \sum_{m=0}^{n}\, \sum_{\tilde m=0}^{n}\, C_{m,\tilde m}\ \frac{1}{(n-m)!(n+1-\tilde m)!} \Big( \frac{1}{(k-m)!(k+1-\tilde m)!} - \frac{1}{(k+1-m)!(k-\tilde m)!} \Big)=\\[0.2cm]
& = C_{k,n} \sum_{m=0}^{n}\, \sum_{\tilde m=m}^{n}\, C_{m,\tilde m}\ \frac{\tilde m - m}{(n-m)!(n+1-\tilde m)! (k+1-m)!(k+1-\tilde m)!} + \\[0.2cm]
& + C_{k,n} \sum_{\tilde m=0}^{n}\, \sum_{m=\tilde m}^{n}\, C_{m,\tilde m}\ \frac{\tilde m - m}{(n-m)!(n+1-\tilde m)! (k+1-m)!(k+1-\tilde m)!} = \\[0.2cm]
\end{aligned}
\]
where in the last equality we have splitted the sum in the square into the two sums below and above the diagonal. At this point we have written our determinant as the sum of two similar terms, and to compare them we change the names in the second sum and notice that $C_{m,\tilde m} = C_{\tilde m, m}$. Hence we write
\[
\begin{aligned}
& \alpha^{+}_{k,n}\, \alpha^{+}_{k+1, n+1}\, - \alpha^{+}_{k+1,n}\, \alpha^{+}_{k,n+1} = \\[0.2cm]
& = C_{k,n} \sum_{m=0}^{n}\, \sum_{\tilde m=m}^{n}\, C_{m,\tilde m}\ \frac{\tilde m - m}{(n-m)!(n+1-\tilde m)! (k+1-m)!(k+1-\tilde m)!} - \\[0.2cm]
& - C_{k,n} \sum_{m=0}^{n}\, \sum_{\tilde m=m}^{n}\, C_{\tilde m, m} \frac{\tilde m - m}{(n-\tilde m)!(n+1- m)! (k+1-\tilde m)!(k+1-m)!} = \\[0.2cm]
& = C_{k,n} \sum_{m=0}^{n}\, \sum_{\tilde m=m}^{n}\, C_{m,\tilde m}\ \frac{\tilde m - m}{(n-m)!(n-\tilde m)! (k+1-m)!(k+1-\tilde m)!}\Big( \frac{1}{n+1-\tilde m} - \frac{1}{n+1-m} \Big)=\\[0.2cm]
&= C_{k,n} \sum_{m=0}^{n}\, \sum_{\tilde m=m}^{n}\, C_{m,\tilde m}\ \frac{(\tilde m - m)^{2}}{(n+1-m)!(n+1-\tilde m)! (k+1-m)!(k+1-\tilde m)!} \ge 0
\end{aligned}
\]

\begin{figure}[h]
\begin{center}
\subfigure[]
    {\includegraphics[width=6cm]{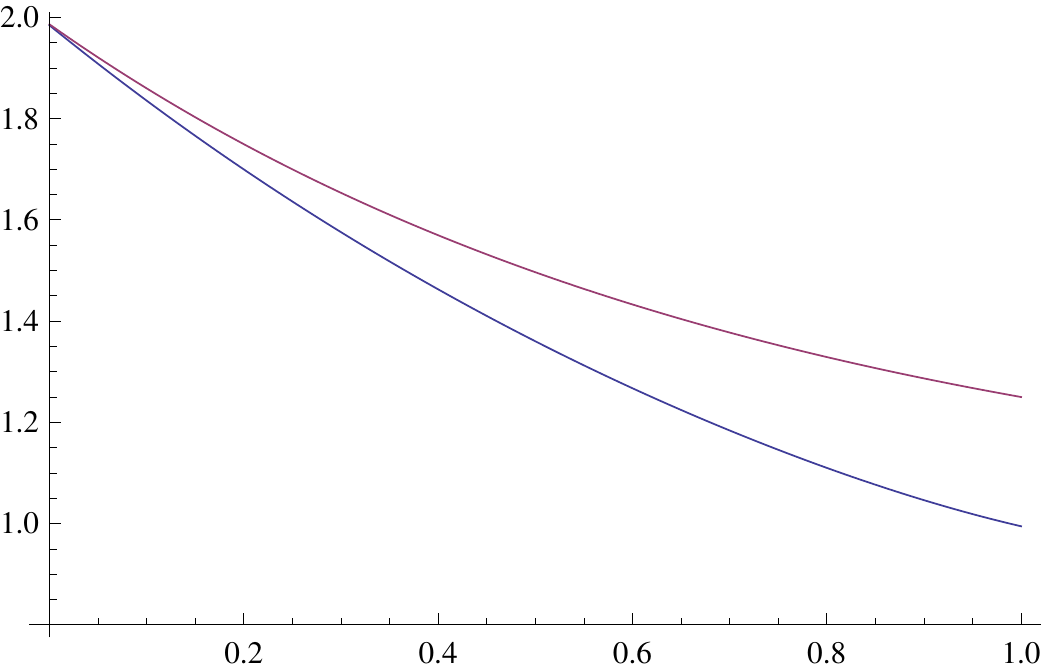}}
    \hspace{0.7cm}
    \subfigure[]
    {\includegraphics[width=6cm]{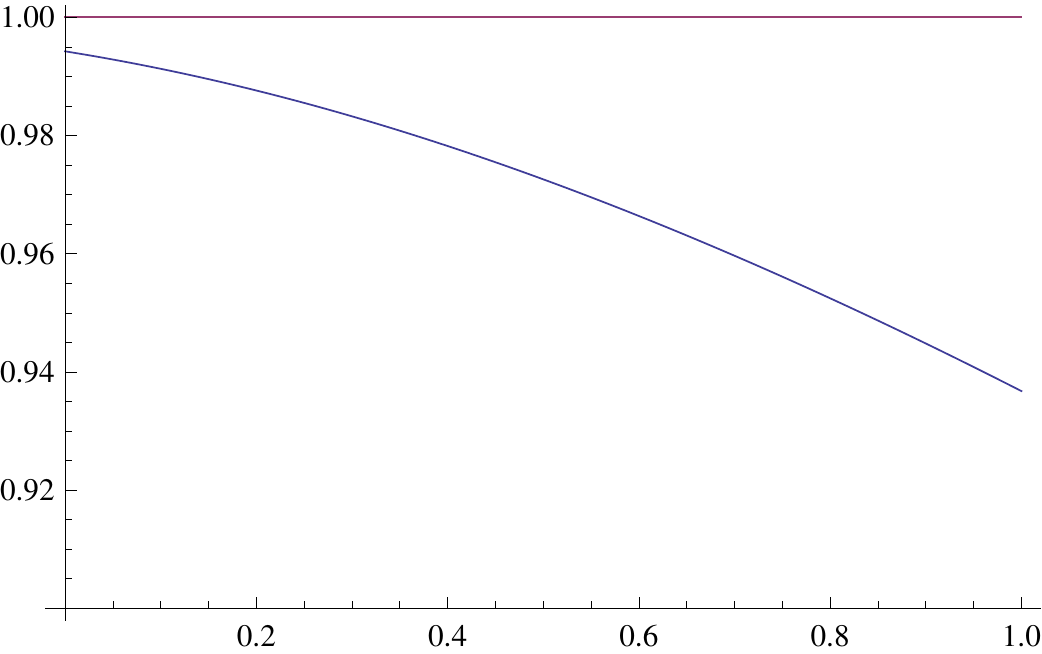}}
\caption{(a) The blue line is the curve $\lambda^{+}_{50}(q)$ and the red one is the function $f^{+}(q) = 1 + 2^{-2q}$; (b) the blue line is the curve $\lambda^{-}_{50}(q)$ and the red one is the function $f^{-}(q)=1$.} \label{confronto}
\end{center}
\end{figure}

\begin{figure}[h]
    \begin{center}
    \subfigure[]
    {\includegraphics[width=6.0cm,height=4.5cm]{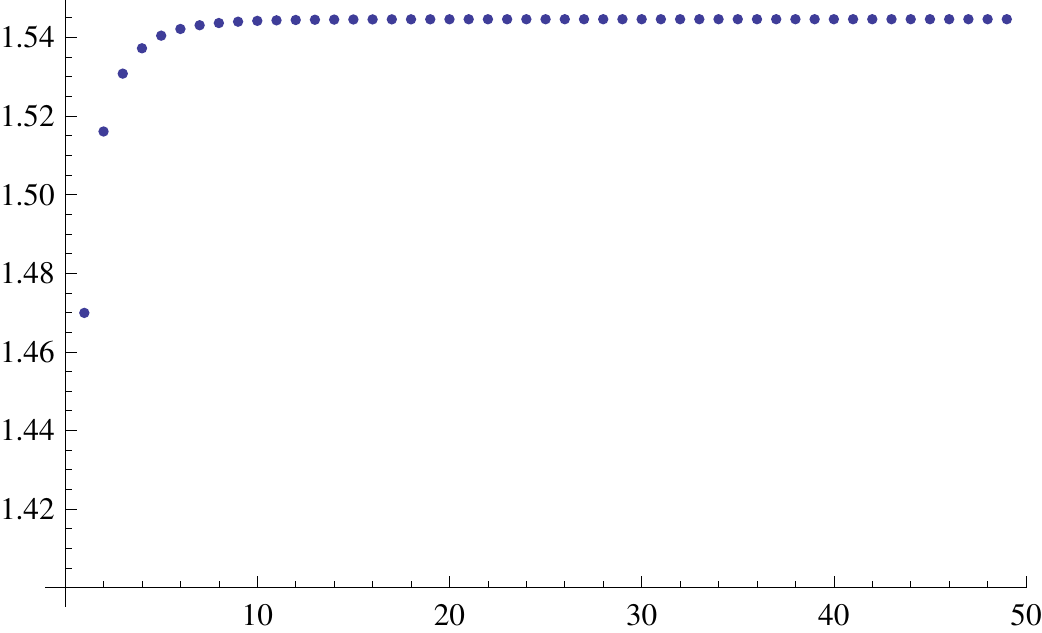}}
    \hspace{0.7cm}
    \subfigure[]
    {\includegraphics[width=6.0cm,height=4.5cm]{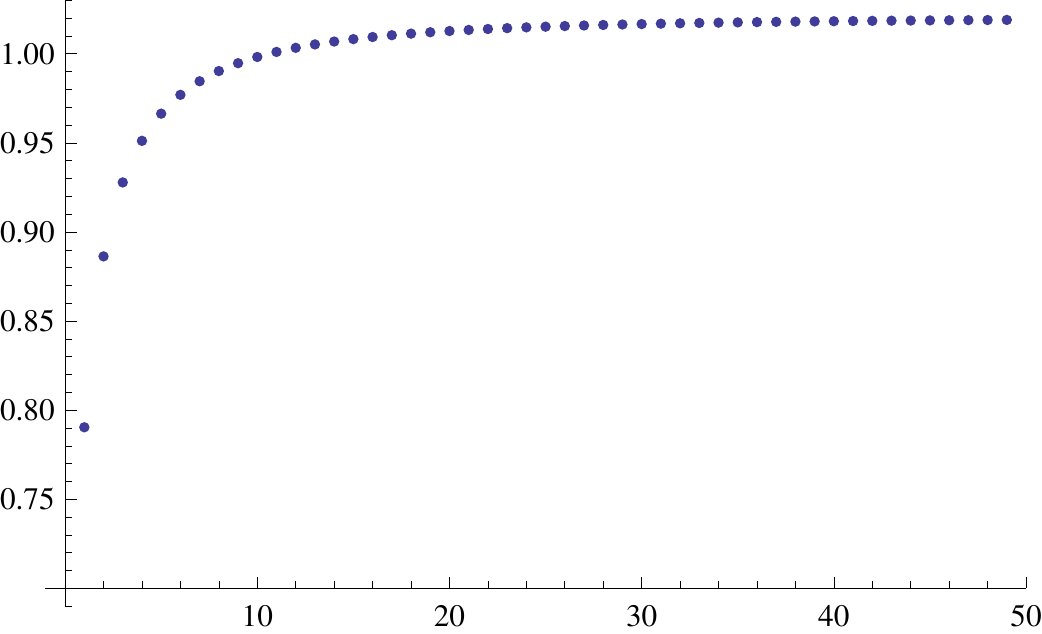}}
    \caption{The sequences $\set{\lambda^{+}_{N}(q)}$ for $q=1/3$ on the left and $q=0.95$ on the right with $N=1,\dots,50$.}
    \label{convergence}
    \end{center}
\end{figure}

\begin{figure}[h]
\begin{center}
\subfigure[]
    {\includegraphics[width=6cm]{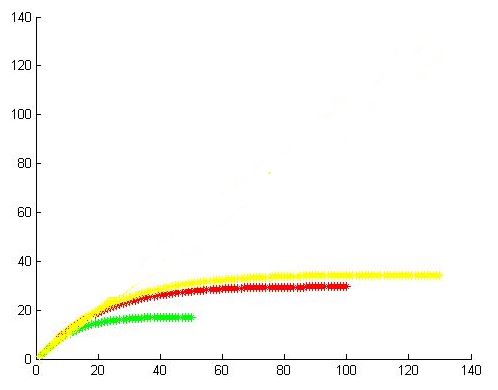}}
    \hspace{0.7cm}
    \subfigure[]
    {\includegraphics[width=6cm]{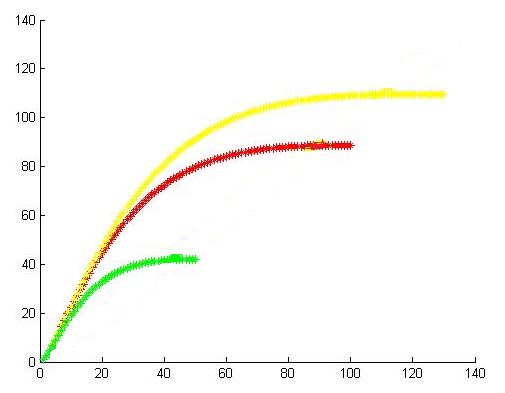}}
\subfigure[]
    {\includegraphics[width=6.0cm,height=4.5cm]{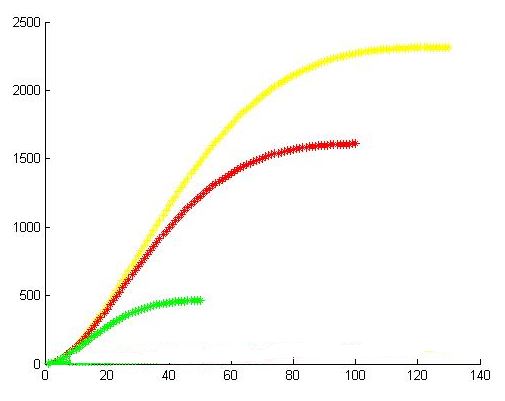}}
    \hspace{0.7cm}
    \subfigure[]
    {\includegraphics[width=6.0cm,height=4.5cm]{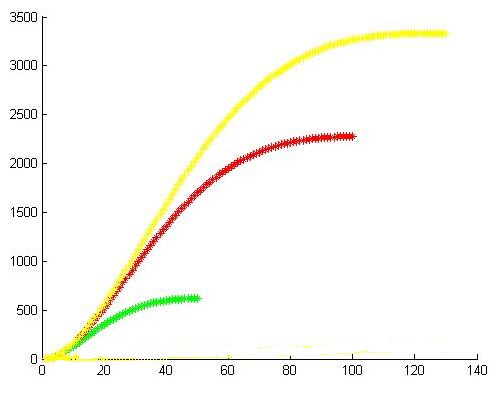}}
    \caption{The curves $S_N(k)$ for $N=50,100,130$ with (a) $q=0.3$; (b) $q=0.5$; (c) $q=0.95$; (d) $q=1$.}
    \label{fig-sums}
    \end{center}
\end{figure}


\end{document}